\RequirePackage{fix-cm}
\documentclass[11pt,letterpaper]{article}
\usepackage{fullpage}

\usepackage{graphicx}
\usepackage{fullpage}
\usepackage{latexsym,graphicx}
\usepackage{amsmath,amssymb,amsthm,enumerate}
\usepackage{color}
\usepackage{bbm,bm}

\newtheorem{theorem}{Theorem}[section]
\newtheorem{definition}[theorem]{Definition}
\newtheorem{proposition}[theorem]{Proposition}

\newtheorem{lemma}[theorem]{Lemma}

\newtheorem{observation}[theorem]{Observation}
\newtheorem{claim}[theorem]{Claim}

\newenvironment{proofof}[1]{\noindent{\emph{Proof of {#1}.}}}{}

\newcommand{\R}[0]{{\ensuremath{\mathbb{R}}}}
\newcommand{\N}[0]{{\ensuremath{\mathbb{N}}}}

\newcommand{\conv}[0]{\ensuremath{\textrm{conv}}}
\newcommand{\dist}[0]{\ensuremath{\textrm{d}}}

\newcommand{\sign}[0]{\ensuremath{\textrm{sign}}}
\newcommand{\Var}[0]{\ensuremath{\textrm{Var}}}
\newcommand{\Cov}[0]{\ensuremath{\textrm{Cov}}}
\newcommand{\E}[0]{\ensuremath{\mathbb{E}}}
\renewcommand{\P}[0]{P}
\newcommand{\Q}[0]{Q}
\newcommand{\X}[0]{\boldsymbol X}
\newcommand{\bcX}[0]{\bm{\mathcal{X}}}
\newcommand{\bP}[0]{\boldsymbol P}
\newcommand{\bd}[0]{\boldsymbol d}
\newcommand{\bZ}[0]{\boldsymbol Z}
\newcommand{\bQ}[0]{\boldsymbol Q}
\newcommand{\A}[0]{\boldsymbol A}
\newcommand{\U}[0]{\boldsymbol U}
\newcommand{\Y}[0]{\boldsymbol Y}
\newcommand{\y}[0]{\boldsymbol y}
\newcommand{\cE}[0]{\mathcal{E}}
\newcommand{\D}[0]{\boldsymbol D}
\newcommand{\W}[0]{\boldsymbol W}

\newcommand{\remove}[1]{}

\newcommand{\roott}[0]{r}
\DeclareMathOperator{\intt}{int}
\DeclareMathOperator{\leftc}{left}
\DeclareMathOperator{\rightc}{right}

\DeclareMathOperator{\proj}{proj}

\newcounter{mynote}[section]

\begin{document}
\def\bred{}

\title{How Good Are Sparse Cutting-Planes?\thanks{Santanu S. Dey and Qianyi Wang were partially supported by NSF grant CMMI-1149400.}}

\author{Santanu S. Dey \and Marco Molinaro \and Qianyi Wang}




\maketitle

\begin{abstract}
Sparse cutting-planes are often the ones used in mixed-integer programing (MIP) solvers, since they help in solving the linear programs encountered during branch-\&-bound more efficiently. However, how well can we approximate the integer hull by just using sparse cutting-planes? In order to understand this question better, given a polyope $\P$ (e.g. the integer hull of a MIP), let $\P^k$ be its best approximation using cuts with at most $k$ non-zero coefficients. We consider $\dist(\P, \P^k) = \max_{x \in \P^k} \left(\min_{y \in \P} \| x - y\|\right)$ as a measure of the quality of sparse cuts. 

In our first result, we present general upper bounds on $\dist(\P, \P^k)$ which depend on the number of vertices in the polytope and exhibits three phases as $k$ increases. Our bounds imply that if $\P$ has polynomially many vertices, using half sparsity already approximates it very well. Second, we present a lower bound on $\dist(\P, \P^k)$ for random polytopes that show that the upper bounds are quite tight. Third, we show that for a class of hard packing IPs, sparse cutting-planes do not approximate the integer hull well, that is $d(\P, \P^k)$ is large for such instances unless $k$ is very close to $n$. Finally, we show that using sparse cutting-planes in extended formulations is at least as good as using them in the original polyhedron, and give an example where the former is actually much better. 

\end{abstract}


\section{Introduction}\label{intro}

Most successful mixed integer linear programming (MILP) solvers are based on branch-$\&$-bound and cutting-plane (cut) algorithms. Since MILPs belong to the class of NP-hard problems, one does not expect the size of branch-$\&$-bound tree to be small (polynomial is size) for every instance. In the case where the branch-$\&$-bound tree is not small, a large number of linear programs must be solved. It is well-known that dense cutting-planes are difficult for linear programming solvers to handle. Therefore, most commercial MILPs solvers consider sparsity of cuts as an important criterion for cutting-plane selection and use \cite{guPC,tobiasPC,narisettyPC}.

Surprisingly, very few studies have been conducted on the topic of sparse cutting-planes. Apart from cutting-plane techniques that are based on generation of cuts from single rows (which implicitly lead to sparse cuts if the underlying row is sparse), to the best of our knowledge only the paper \cite{AndersenWeismantel} explicitly discusses methods to generate sparse cutting-planes. 

The use of sparse cutting-planes may be viewed as a compromise between two competing objectives. As discussed above, on the one hand, the use of sparse cutting-planes aids in solving the linear programs encountered in the branch-$\&$-bound tree faster. On the other hand, it is possible that `important' facet-defining or valid inequalities for the convex hull of the feasible solutions are dense and thus without adding these cuts, one may not be able to attain significant integrality gap closure. This may lead to a larger branch-$\&$-bound tree and thus result in the solution time to increase. 

It is challenging to simultaneously study both the competing objectives in relation to cutting-plane sparsity. Therefore, a first approach to understanding usage of sparse cutting-planes is the following:  {If we are able to separate and use valid inequalities with a given level of sparsity (as against completely dense cuts), how much does this cost in terms of loss in closure of integrality gap?} 

Considered more abstractly, the problem reduces to a purely geometric question: Given a polytope $\P$ (which represents the convex hull of feasible solutions of a MILP), how well is $\P$ approximated by the use of sparse valid inequalities. In this paper we will study polytopes contained in the $[0,\ 1]^n$ hypercube. This is without loss of generality since one can always translate and scale a polytope to be contained in the $[0, \ 1]^n$ hypercube.
 
\vspace{-3pt}
\subsection{Preliminaries} \label{sec:prelim}
		
			A cut $a x \le b$ is called \emph{$k$-sparse} if the vector $a$ has at most $k$ nonzero components. Given a set $\P \subseteq \R^n$, define $\P^k$ as the best outer-approximation obtained from $k$-sparse cuts, that is, it is the intersection of all $k$-sparse cuts valid for $\P$. 
			
			For integers $k$ and $n$, let $[n]:= \{1, \dots, n\}$ and let $[n] \choose k$ be the set of all subsets of $[n]$ of cardinality $k$. Given a $k$-subset of indices $I \subseteq [n]$, define $\R^{\bar{I}} = \{x \in \R^n : x_i = 0 \textrm{ for all } i \in I\}$. An equivalent and handy definition of $\P^k$ is the following: $\P^k = \bigcap_{I \in {[n] \choose k}} \left(\P + \R^{\bar{I}}\right).$ Thus, if $\P$ is a polytope, {\bred then $\P^k$ is also a polytope.}

\vspace{-3pt}
\subsection{Measure of Approximation}
There are several natural measures to compare the quality of approximation provided by $\P^k$ in relation to $\P$.  For example, one may consider objective value ratio: maximum over all costs $c$ of expression $\frac{z^{c,k}}{z^c}$, where $z^{c,k}$ is the value of maximizing $c$ over $\P^k$, and $z^c$ is the same for $\P$. We discard this ratio, since this ratio can become  infinity and not provide any useful information. {\bred For example take $\P = \conv\{(0,0), (0, 1), (1, 1)\}$ and compare with $\P^1$ wrt $c = (1, -1)$}. Similarly, we may compare  the volumes of $\P$ and $\P^k$. However, this ratio is not useful if $\P$ is not full-dimensional and $\P^k$ is. 

In order to have a useful measure that is well-defined for all polytopes contained in $[0, 1]^n$, we consider the following \emph{distance measure}:
	\vspace{-4pt}
	\begin{align*}
		\dist(\P, \P^k):= \max_{x \in \P^k} \left(\min_{y \in \P} \| x - y\|\right),
	\end{align*} 
where $\| \cdot \|$ is the $\ell_2$ norm. It is easily verified that there is a vertex of $\P^k$ attaining the maximum above. Thus, alternatively the distance measure can be interpreted as the Euclidean distance between $\P$ and the farthest vertex of $\P^k$ from $\P$. 

\begin{observation}[$\dist(\P, \P^k)$ is an upper bound on depth of cut] Suppose $\alpha x \leq \beta$ is a valid inequality for $\P$ where $\|\alpha\| = 1$. Let the \emph{depth} of this cut be the smallest $\gamma \geq 0$ such that $\alpha x \leq \beta + \gamma$ is valid for $\P^k$. It is straightforward to verify that $\gamma \leq \dist(\P, \P^k)$. Therefore, the distance measure gives an upper bound on additive error when optimizing a (normalized) linear function over $\P$ and $\P^k$. 
\end{observation}

\begin{observation}[Comparing $\dist(\P, \P^k)$ to $\sqrt{n}$] Notice that the largest distance between any two points in the $[0,\ 1 ]^n$ hypercube is at most $\sqrt{n}$. Therefore in the rest of the paper we will compare the value of $\dist(\P, \P^k)$ to $\sqrt{n}$. 
\end{observation}


\subsection{Some Examples}

In order to build some intuition we begin with some examples in this section. Let $\P := \{ x \in [0, \ 1]^n: ax \leq b\}$ where $a$ is a non-negative vector. It is straightforward to verify that in this case, $\P^k: = \{  x \in [0, \ 1]^n: a^Ix \leq b \ \forall I \in {[n]\choose k}\}$, where $a^I_j:= a_j $ if $j \in I$ and $a^I_j = 0$ otherwise.

\vspace{-5pt}
\paragraph{Example 1:} Consider the simplex $\P = \{x \in [0,1]^n : \sum_{i =1}^n x_i \le 1\}$. Using the above observation, we have that $\P^k = \conv\{e^1, e^2, \dots, e^n, \frac{1}{k}e\}$, where $e^j$ is the unit vector in the direction of the $j^{th}$ coordinate and $e$ is the all ones vector. Therefore the distance measure between $\P$ and $\P^k$ is $\sqrt{n} (\frac{1}{k} - \frac{1}{n}) \approx \frac{\sqrt{n}}{k}$, attained by the points $\frac{1}{n} e \in \P$ and $\frac{1}{k} e \in \P^k$. This is quite nice because with $k \approx \sqrt{n}$ (which is pretty reasonably sparse) we get a constant distance.  Observe also that the \emph{rate of change of the distance measure} follows a `single pattern' - we call this a \emph{single phase example}.  See Figure \ref{fig1}(a) for $\dist(\P, \P^k)$ plotted against $k$ (in blue) and $k \cdot \dist(\P, \P^k)$ plotted against $k$ (in green).

\vspace{-10pt}
\begin{figure}[h]
\centering
\includegraphics[width=0.9\textwidth, height=1.35in]{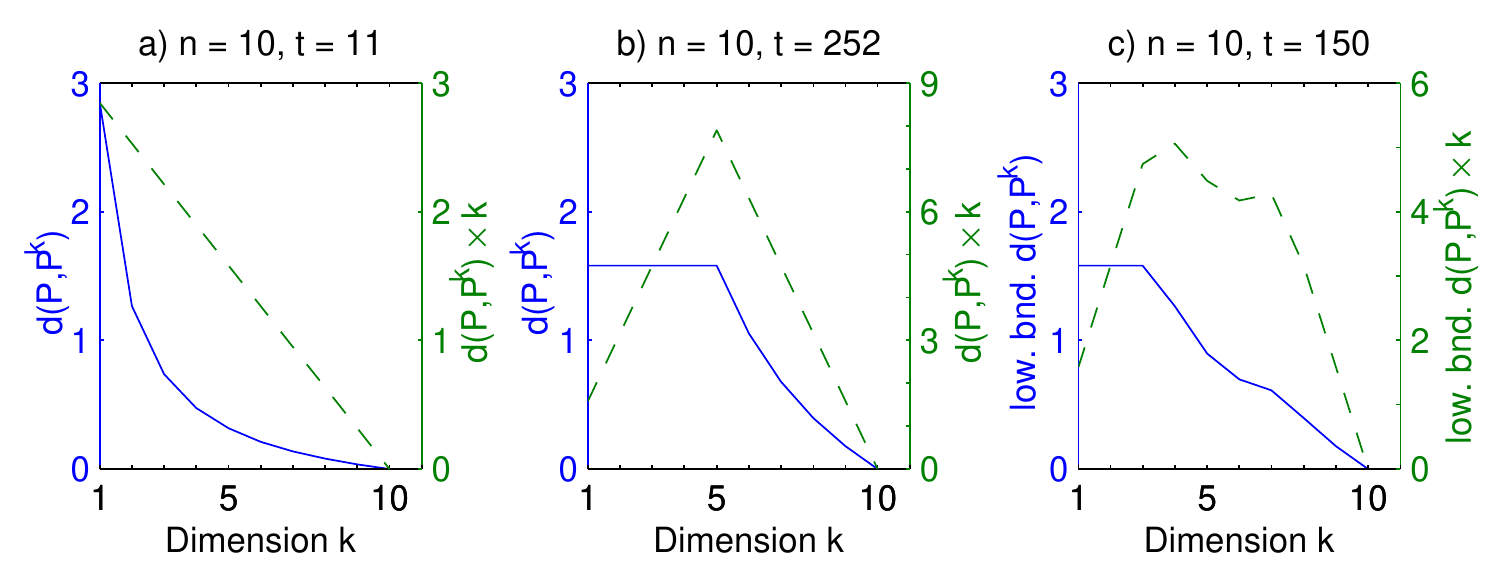}
\caption{(a) Sparsity is good. (b) Sparsity is not so good. (c) Example with three phases.} \label{fig1}
\end{figure}

\vspace{-15pt}
\paragraph{Example 2:} Consider the set $\P = \{x \in [0,1]^n : \sum_i x_i \le \frac{n}{2}\}$. We have that $\P^k: = \{  x \in [0, \ 1]^n: \sum_{i \in I}x_i \leq \frac{n}{2}, \ \forall I \in {[n]\choose k}\}$. Therefore, for all $k \in \{1, \dots, n/2\}$ we have $\P^k = [0, \ 1]^n$ and hence $\dist(\P, \P^k) = \sqrt{n}/2$. Thus, we stay with distance $\Omega(\sqrt{n})$ (the worst possible for polytopes in $[0,1]^n$) even with $\Theta(n)$ sparsity. Also observe that for $k  > \frac{n}{2}$, we have 
$\dist(\P, \P^k) = \frac{n\sqrt{n}}{2k} - \frac{\sqrt{n}}{2}$. Clearly the {rate of change of the distance measure} has \emph{two phases}, first phase of $k$ between $1$ and $\frac{n}{2}$ and the second phase of $k$ between $\frac{n}{2}$ and $n$.  See Figure \ref{fig1}(b) for the plot of $\dist(\P, \P^k)$ against $k$ (in blue) and of $k \cdot \dist(\P, \P^k)$ against $k$ (in green).

\vspace{-5pt}
\paragraph{Example 3:} We present an experimental example in dimension $n = 10$. The polytope $\P$ is now set as the convex hull of $150$ binary points randomly selected from the hyperplane $\{x \in \R^{10} : \sum_{i = 1}^{10} x_i = 5\}$. We experimentally computed lower bounds on $\dist(\P, \P^k)$ which are plotted in Figure \ref{fig1}(c) as the blue line (for details on this computation see Section \ref{app:empiricalLB} of the appendix). Notice that there are now \emph{three phases}, which are more discernible in the plot between the lower bound on $k \cdot \dist(\P, \P^k)$ and $k$ (in green).

The above examples serve to illustrate the fact that different polytopes, behave very differently when we try and approximate them using sparse inequalities. 

\vspace{-5pt}
\section{Main Results} 

\subsection{Upper Bounds}\label{main:up}
Surprisingly, it appears that the complicated  behavior of $\dist(\P, \P^k)$ as $k$ changes can be described to some extent in closed form. Our first result is {\bred a nontrivial} upper bounds on $\dist(\P, \P^k)$ for general polytopes. This {\bred result} is proven in Section \ref{sec:upbnd}. 

\begin{theorem}[Upper Bound on $\dist(\P, \P^k)$]\label{thm:upperall}
Let $n \geq 2$. Let $\P \subseteq [0, 1]^n$ be the convex hull of points $\{p^1, \dots, p^t\}$. Then 
\begin{enumerate}
\item $\dist(\P, \P^k) \leq 4 \max\left\{\frac{n^{1/4}}{\sqrt{k}} \sqrt{8 \max_{i \in [t]} \|p^i\|} \sqrt{\log 4 t n}, \frac{8 \sqrt{n}}{3 k} \log 4tn \right\}$

\item $\dist(\P, \P^k) \leq 2\sqrt{n}\left(\frac{n}{k} -1\right)$.
\end{enumerate}
\end{theorem}
Since $\max_{i \in \{1, \dots, t\}}||p^i|| \leq \sqrt{n}$ and the first upper bound yields nontrivial values only when {\bred $k \geq \frac{32}{3} \log 4 tn$}, a simpler (although weaker) expression for the first  upper bound is {\bred $8 \sqrt{2}\frac{\sqrt{n}}{\sqrt{k}}\sqrt{\log 4 tn}$} \footnote{If $k \geq \frac{8\log 4tn }{9}$, then $\frac{n^{\frac{1}{4}}\sqrt{8\sqrt{n}}\sqrt{\log 4tn} }{\sqrt{k}} \geq \frac{8\sqrt{n}\log 4tn}{3k}$ }.  We make two observations based on Theorem \ref{thm:upperall}. 

Consider polytopes with `few' vertices, say $n^q$  vertices for some constant $q$. Suppose we decide to use cutting-planes with half sparsity (i.e. $k = \frac{n}{2}$), a reasonable assumption in practice. Then plugging in these values, it is easily verified that {\bred $\dist(\P, \P^k) \leq {16}\sqrt{(q + 1)\log n} \approx c \sqrt{\log n}$} for a constant $c$, which is a significantly small quantity in comparison to $\sqrt{n}$. In other words, \emph{if the number of vertices is small, independent of the location of the vertices, using half sparsity cutting-planes allows us to approximate the integer hull very well.} We believe that as the number of vertices increase, the structure of the polytope becomes more important in determining $\dist(\P, \P^k)$ and Theorem \ref{thm:upperall} only captures the worst-case scenario. Overall, Theorem \ref{thm:upperall} presents a theoretical justification for the use of sparse cutting-planes in many cases.  

Theorem \ref{thm:upperall} supports the existence of  three phases in the behavior of $\dist(\P, \P^k)$ as $k$ varies: \textbf{(Small $k$)} When {\bred $k \le 128 \log 4tn$} the (simplified) upper bounds are larger than $\sqrt{n}$, indicating that `no progress' is made in approximating the shape of $\P$ (this is seen Examples 2 and 3). \textbf{(Medium $k$)} When {\bred $128 \log 4tn \le k \lesssim n - \sqrt{n \log 4tn}$} the first upper bound in Theorem \ref{thm:upperall} dominates. \textbf{(Large $k$)} When $k \gtrsim n - \sqrt{n \log 4tn}$ the upper bound $2\sqrt{n}\left(\frac{n}{k} -1\right)$ dominates. In particular,  in this phase, $k \cdot \dist(\P, \P^k) \leq 2n^{3/2} - 2\sqrt{n} k$, i.e., the upper bound times $k$  is a linear function of $k$. All the examples in Section \ref{intro} illustrate this behaviour.

\subsection{Lower Bounds}\label{main:low}
How good is the quality of the upper bound presented in Theorem \ref{thm:upperall}? Let us first consider the second upper bound in Theorem \ref{thm:upperall}. Then observe that for the second example in Section $\ref{intro}$, this upper bound is tight up to a constant factor for $k$ between the values of $\frac{n}{2}$ and ${n}$.

We study lower bounds on $\dist(\bP, \bP^k)$ for random 0/1 polytopes in Section \ref{sec:lowbnd} that show that the first upper bound in Theorem \ref{thm:upperall} is also quite tight.

\begin{theorem} \label{thm:LBRandom01}
{\bred Let $k, t, n \in \mathbb{Z}_{+}$ satisfying $64 \leq k \leq n$ and $(0.5k^2 \log n + 2k + 1)^2 \le t \le e^n$.} Let $\X^1, \X^2, \ldots, \X^t$ be independent uniformly random points in $\{0,1\}^n$ and let $\bP = \conv(\X^1, \X^2, \ldots, \X^t)$. Then with probability at least $1/4$ we have that $$\dist(\bP, \bP^k) \ge \min\left\{\frac{\sqrt{n}}{\sqrt{k}} \frac{\sqrt{\log t}}{110 \sqrt{\log n}}, \frac{\sqrt{n}}{8}\right\} \left(\frac{1}{2} - \frac{1}{k^{3/2}} \right) - 3 \sqrt{\log t}.$$
\end{theorem}
	
	Let us compare this lower bound with the simpler expression {\bred $8\sqrt{2}\frac{\sqrt{n}}{\sqrt{k}}\sqrt{\log tn}$} for the first part of the upper bound of Theorem \ref{thm:upperall}. We focus on the case where the minimum in the lower bound is achieved by the first term. Then comparing the leading term {\bred $\sqrt{\frac{n}{k}} \frac{\sqrt{\log t}}{2 \cdot 110 \sqrt{\log n}}$} in the lower bound with the upper bound, we see that these quantities match up to a factor of {\bred $O\big(\frac{ \sqrt{\log(tn)} \sqrt{\log n}}{ \sqrt{\log t}}\big)$}, showing that for many $0/1$ polytopes the first upper bound of Theorem \ref{thm:upperall} is quite tight. We also remark that in order to simplify the exposition we did not try to optimize constants and lower order terms in our bounds. 

The main technical tool for proving this lower bound is a new anticoncentration result for linear combinations $a\X$, where the $\X_i$'s are independent Bernoulli random variables (Lemma \ref{lemma:anticoncentration}). The main difference from standard anticoncentration results is that the latter focus on variation around the standard deviation; in this case, standard tools such as the Berry-Esseen Theorem or the Paley-Zygmund Inequality\cite{anirban}
 can be used to obtain constant-probability anticoncentration. However, we need to control the behavior of $a\X$ much further {\bred away} from its standard deviation, where we cannot hope to get constant-probability anticoncentration.


	\vspace{-5pt}
	\subsection{Hard Packing Integer Programs}
	
		We also study well-known, randomly generated, hard packing integer program instances (see for instance \cite{letchford}). Given parameters $n, m, M \in \N$, the convex hull of the packing IP is given by $\bP = \conv(\{x \in \{0,1\}^n : \A^j x \le \frac{\sum_{i} \A^j_i}{2}, ~\forall j \in [m]\})$, where the $\A^j_i$'s are chosen independently and uniformly in the set $\{0, 1, \ldots, M\}$. Let $(n, m, M)$-PIP denote the distribution over the generated $\P$'s.

	The following result shows the limitation of sparse cuts for these instances. 
	
	\begin{theorem} \label{thm:PIPLB}
		Consider $n, m, M \in \N$ such that $n \ge 50$ and $8 \log 8n \le m \le n$. Let $\bP$ be sampled from the distribution $(n, m, M)$-PIP. Then with probability at least $1/2$, $\dist(\bP, \bP^k) \ge \frac{\sqrt{n}}{2} \left(\frac{2}{\max\{\alpha,1\}} (1 - \epsilon)^2 - (1 + \epsilon') \right),$ where $c = k/n$ and 
		\begin{gather*}
			\frac{1}{\alpha} = \frac{M}{2(M+1)} \left[\frac{n - 2 \sqrt{n \log 8 m}}{c ((2-c) n + 1) + 2 \sqrt{10 c nm}} \right], ~~~~ \epsilon = \frac{24 \sqrt{\log 4 n^2 m}}{\sqrt{n}}, \\
			\epsilon' = \frac{3 \sqrt{\log 8n}}{\sqrt{m} - 2 \sqrt{\log 8n}}.
		\end{gather*}
	\end{theorem}		
	
	Notice that when $m$ is sufficiently large, and $n$ reasonably larger than $m$, we have $\epsilon$ and $\epsilon'$ approximately 0, and the above bound reduces to approximately $\frac{\sqrt{n}}{2} \left(\left(\frac{M}{M+1}\right) \left(\frac{n}{k (2-n/k)}\right) - 1\right) \approx \frac{\sqrt{n}}{2} \left(\frac{n}{k (2-n/k)} - 1\right)$, which is within a constant factor of the upper bound from Theorem \ref{thm:upperall}. The poor behavior of sparse cuts 
	gives an indication for the hardness of these instances and suggests that denser cuts should be explored in this case. 
	
	One interesting feature of this result is that it works directly with the IP formulation, not relying on an explicit linear description of the convex hull.

	
	\vspace{-5pt}
	\subsection{Sparse Cutting-Planes and Extended Formulations}	
	
	 Let $\proj_x : \R^{n} \times \R^{m} \rightarrow \R^n$ denote the projection operator onto the first $n$ coordinates. We say that a set $\Q \subseteq \R^n \times \R^m$ is an \emph{extended formulation} of $\P \subseteq \R^n$ if $\P = \proj_x(\Q)$. 
	
	 As our final result we remark that using sparse cutting-planes in extended formulations is at least as good as using them in the original polyhedron, and sometime much better. These results are proved in Section \ref{sec:exended}.
	 
\begin{proposition} \label{lemma:extFormContainment}
	Consider a polyhedron $\P \subseteq \R^n$ and an extended formulation $\Q \subseteq \R^n \times \R^m$ for it. Then $\proj_x(\Q^k) \subseteq (\proj_x (\Q))^k = \P^k$.
\end{proposition}

	\begin{proposition} \label{lemma:goodExtFormulation}
		Consider $n \in \N$ and assume it is a power of 2. Then there is a polytope $\P \subseteq \R^n$ such that: 
\begin{enumerate}
\item $\dist(\P, \P^k) = \sqrt{n/2}$ for all $k \le n/2$.
\item There is an extended formulation $\Q \subseteq \R^{n} \times \R^{2n - 1}$ of $\P$ such that $\proj_x(\Q^3) = \P$. 
\end{enumerate}
	\end{proposition}


\vspace{-5pt}
\section{Upper Bound}\label{sec:upbnd}

	In this section we prove Theorem \ref{thm:upperall}. In fact we prove the same bound for polytopes in $[-1,1]^n$, which is a slightly stronger result. The following well-known property is crucial for the constructions used in both parts of the theorem.
	
	\begin{observation}[Section 2.5.1 of \cite{boydVanden}] \label{obs:sepClosesPoint}
		Consider a compact convex set $S \subseteq \R^n$. Let $\bar{x}$ be a point outside $S$ and let $\bar{y}$ be the closest point to $\bar{x}$ in $S$. Then setting $a = \bar{x} - \bar{y}$, the inequality $ax \le a\bar{y}$ is valid for $S$ and cuts $\bar{x}$ off.   
	\end{observation}


	\subsection{Proof of First Part of Theorem \ref{thm:upperall}}

		Consider a polytope $\P = \conv\{p^1, p^2, \ldots, p^t\}$ in $[-1,1]^n$. Define $$\lambda^* = \max\left\{\frac{n^{1/4}}{\sqrt{k}} \sqrt{8 \max_i \|p^i\|} \sqrt{\log 4 t n}, \frac{8 \sqrt{n}}{3 k} \log 4tn \right\}.$$ 
		In order to show that $\dist(\P, \P^k)$ is at most $4 \lambda^*$ we show that every point at distance more than $4 \lambda^*$ from $\P$ is cut off by a valid inequality for $\P^k$. Assume until the end of this section that $4 \lambda^*$ is at most $\sqrt{n}$, otherwise the result is trivial; in particular, this implies that the second term in the definition of $\lambda^*$ is at most $\sqrt{n}/4$ and hence $k \ge 8 \log 4tn$ {\bred (In fact, $k \ge \frac{32}{3} \log 4tn$, but  $k \ge 8 \log 4tn$ suffices for the rest of the proof).}
		
		So let $u \in \R^n$ be a point at distance more than $4 \lambda^*$ from $\P$. Let $v \in \P$ be the closest point in $\P$ to $\P^k$. We can write $u = v + \lambda d$ for some vector $d$ with $\|d\|_2 = 1$ and $\lambda > 4\lambda^*$. From Observation \ref{obs:sepClosesPoint}, inequality $dx \le dv$ is valid for $\P$, so in particular $d p^i \le dv$ for all $i \in [t]$; in addition, it that this inequality cuts off $u$: $du = dv + \lambda > dv$. The idea is to use this extra slack factor $\lambda$ in the previous equation to show we can `sparsify' the inequality $dx \le dv$ while maintaining separation of $\P$ and $u$. It then suffices to prove the following lemma.
		
		\begin{lemma} \label{lemma:existsSep}
			There is a $k$-sparse vector $\tilde{d} \in \R^n$ such that 
				{\bred
				\begin{enumerate}
					\item $\tilde{d} p^i \le \tilde{d} v + \frac{\lambda}{2}$, for all $i \in [t]$
					\item $\tilde{d} u > \tilde{d} v + \frac{\lambda}{2}$.
				\end{enumerate}}
		\end{lemma}
		
		To prove the lemma we construct a random vector $\tilde{\D} \in \R^n$ which, with non-zero probability, is $k$-sparse and satisfies the two other requirements of the lemma. Let $\alpha = \frac{k}{2 \sqrt{n}}$. Define $\tilde{\D}$ as the random vector with independent coordinates, where $\tilde{\D}_i$ is defined as follows: if $\alpha |d_i| \ge 1$, then $\tilde{\D}_i = d_i$ with probability 1; if $\alpha |d_i| < 1$, then $\tilde{\D}_i$ takes value $\sign(d_i)/\alpha$ with probability $\alpha |d_i|$ and takes value $0$ with probability $1 - \alpha |d_i|$. (For convenience we define $\sign(0) = 1$.) 
			
		The next {\bred proposition} follows directly from the definition of $\tilde{\D}$.	
		
		\begin{proposition} \label{obs:dTilde}
			For every vector $a \in \R^n$ the following hold:
			\begin{enumerate}
				\item $\E[\tilde{\D}a] = da$
				\item $\Var(\tilde{\D}a) \le \frac{1}{\alpha} \sum_{i \in [n]} a_i^2 |d_i|$
				\item $|\tilde{\D}_i a_i - \E[\tilde{\D}_i a_i]| \le \frac{|a_i|}{\alpha}.$ 
			\end{enumerate}
		\end{proposition}
			
		\begin{claim}
			With probability at least $1-1/4n$, $\tilde{\D}$ is $k$-sparse.
		\end{claim}

		\begin{proof}
			Construct the vector $a \in \R^n$ as follows: if $\alpha |d_i| \ge 1$ then $a_i = 1/d_i$, and if $\alpha |d_i| < 1$ then $a_i = \alpha/\sign(d_i)$. Notice that $\tilde{\D} a$ equals the number of non-zero coordinates of $\tilde{\D}$ and $\E[\tilde{\D} a] \le \alpha \|d\|_1 \le k/2$. {\bred Here the first inequality follows from the fact that $E(\tilde{D}_i a_i) \leq \alpha |d_i|$ for all $i$ and the second inequality follows from the definition of $\alpha$ and the fact that $||d||_2 = 1$.} Also, from {\bred Proposition \ref{obs:dTilde}} we have $$\Var(\tilde{\D} a) \le \frac{1}{\alpha} \sum_{i \in [n]} a_i^2 |d_i| \le \alpha \|d\|_1 \le \frac{k}{2}.$$ Then using Bernstein's inequality (Section \ref{app:bernstein} of the appendix) we obtain
			{\bred
			\begin{align*}
				\Pr(\tilde{\D} a > k) \le \exp\left(- \min \left\{\frac{k^2}{8k}, \frac{3k}{8} \right\} \right) \le \frac{1}{4n},
			\end{align*}}
			where the last inequality uses our assumption that $k \ge 8 \log 4 t n$.
		\end{proof}
		
		We now show that property 1 required by Lemma \ref{lemma:existsSep} holds for $\tilde{\D}$ with high probability. {\bred Since $d(p^i - v) \leq 0$ for all $i \in [t]$, the folloing claim shows that property 1 holds with probability at least $1 - \frac{1}{4}$.}
		
		\begin{claim} \label{claim:supProc}
			$\Pr(\max_{i \in [t]} [\tilde{\D} (p^i - v) - d(p^i - v)] > 2 \lambda^*) \le 1/4n$.  
		\end{claim}
		
		\begin{proof}
			Define the centered random variable $\bZ = \tilde{\D} - d$. To make the analysis cleaner, notice that $\max_{i \in [t]} \bZ (p^i - v) \le 2 \max_{i \in [t]} |\bZ p^i|$; this is because $\max_{i \in [t]} \bZ (p^i - v) \le \max_{i \in [t]} |\bZ p^i| + |\bZ v|$, and because for all $a \in \R^n$ we have $|a v| \le \max_{p \in \P} |a p| = \max_{i \in [t]} |a p^i|$ (since $v \in \P$).
			
			Therefore our goal is to upper bound the probability that the process $\max_{i \in [t]} |\bZ p^i|$ is larger then $\lambda^*$. Fix $i \in [t]$. By Bernstein's inequality,
			{\bred
			\begin{align}
				\Pr(|\bZ p^i| > \lambda^*) \le \exp\left(-\min\left\{\frac{(\lambda^*)^2}{4 \Var(|\bZ p^i|)}, \frac{3 \lambda^*}{4 M}\right\}\right), \label{eq:supProcBern}
			\end{align}}
			where $M$ is an upper bound on $\max_j |\bZ_j p^i_j|$.
			
			To bound the terms in the right-hand side, from {\bred Proposition \ref{obs:dTilde}} we have $$\Var(\bZ p^i) = \Var(\tilde{\D} p^i) \le \frac{1}{\alpha} \sum_j (p^i_j)^2 |d_j| \le \frac{1}{\alpha} \sum_j p^i_j |d_j| \le \frac{1}{\alpha} \|p^i\| \|d\| = \frac{1}{\alpha} \|p^i\|,$$ where the second inequality follows from the fact $p^i \in [0,1]^n$, and the third inequality follows from the Cauchy-Schwarz inequality. Moreover, it is not difficulty to see that for every random variable $\W$, $\Var(|\W|) \le \Var(\W)$. Using the first term in the definition of $\lambda^*$, we then have 
			\begin{align*}
				\frac{(\lambda^*)^2}{\Var(|\bZ p^i|)} \ge 4 \log 4 tn.
			\end{align*}
			In addition, for every coordinate $j$ we have $|\bZ_j p^i_j| = |\tilde{\D}_j p^i_j - \E[\tilde{\D}_j p^i_j]| \le 1/\alpha$, where the inequality follows from {\bred Proposition \ref{obs:dTilde}}. Then we can set $M = 1/\alpha$ and using the second term in the definition of $\lambda^*$ we get $\frac{\lambda^*}{M} \ge \frac{4}{3} \log 4 tn$.
			Therefore, replacing these bounds in inequality \eqref{eq:supProcBern} gives $\Pr(|\bZ p^i| \ge \lambda^*) \le \frac{1}{4tn}.$
			
			Taking a union bound over all $i \in [t]$ gives that $\Pr(\max_{i \in [t]} |\bZ p^i| \ge \lambda^*) \le 1/4n$. This concludes the proof of the claim. 
		\end{proof}

		\begin{claim}
			$\Pr(\tilde{\D}(u - v) \le \lambda/2) \le 1 - 1/(2n - 1)$. 
		\end{claim}
		
		\begin{proof}
			Recall $u - v = \lambda d$, hence it is equivalent to bound $\Pr(\tilde{\D}d \le 1/2)$. First, $\E[\tilde{\D}d] = dd = 1$. Also, from {\bred Proposition \ref{obs:dTilde}} we have $\tilde{\D}d \le |\tilde{\D}d - dd| + |dd| \le \frac{1}{\alpha} \sum_i |d_i| + 1 \le \frac{2n}{k} + 1 \le n$, where the last inequality uses the assumption $k \ge 8 \log 4tn$. Then employing Markov's inequality to the non-negative random variable $n - \tilde{\D}d$, we get $\Pr(\tilde{\D}d \le 1/2) \le 1 - \frac{1}{2n - 1}$.
			This concludes the proof.	
		\end{proof}
		
		\begin{proofof}{Lemma \ref{lemma:existsSep}}
			Employ the previous three claims and union bound to find a realization of $\tilde{\D}$ that is $k$-sparse and satisfies requirements 1 and 2 of the lemma.
		\end{proofof}
		\smallskip
		
		This concludes the proof of the first part of Theorem \ref{thm:upperall}.

		\begin{observation}
			Notice that in the above proof $\lambda^*$ is set by Claim \ref{claim:supProc}, and need to be essentially $\E[\max_{i \in [t]} (\tilde{\D} - d) p^i]$. There is a vast literature on bounds on the supremum of stochastic processes (see for instance \cite{kol}), and improved bounds for structured $\P$'s are possible (for instance, via the \emph{generic chaining} method).
		\end{observation}

	
	\subsection{Proof of Second Part of Theorem \ref{thm:upperall}}
		
			The main tool for proving this upper bound is the following lemma, which shows that when $\P$ is `simple', and we have a stronger control over the distance of a point $\bar{x}$ to $\P$, then there is a $k$-sparse inequality that cuts $\bar{x}$ off.
			
			\begin{lemma} \label{lemma:sepOneHyp}
				Consider a {\bred halfspace} $H = \{x \in \R^n : ax \le b\}$ and let $\P = H \cap [-1,1]^n$. Let $\bar{x} \in [-1,1]^n$ be such that $\dist(\bar{x}, H) > 2 \sqrt{n} (\frac{n}{k} - 1)$. Then $\bar{x} \notin \P^k$.
			\end{lemma}
			
			\begin{proof}
				Assume without loss of generality that $\|a\|_2 = 1$. Let $\bar{y}$ be the point in $H$ {\bred closest} to $\bar{x}$, and notice that $\bar{x} = \bar{y} + \lambda a$ where $\lambda > \sqrt{n} (\frac{n}{k} - 1)$. 
				
				For any set $I \in {[n] \choose k}$, the inequality $\sum_{i \in I} a_i x_i \le b + \sum_{i \notin I : a_i \ge 0} a_i - \sum_{i \notin I : a_i < 0} a_i$ is valid for $\P$; since it is $k$-sparse, it is also valid for $\P^k$. Averaging out {\bred these inequalities} over all $I \in {[n] \choose k}$, we get that the following is valid for $\P^k$:
				\begin{align*}
					\frac{k}{n} ax \le b + \left(1- \frac{k}{n}\right) \left( \sum_{i  : a_i \ge 0} a_i - \sum_{i : a_i < 0} a_i \right) \equiv ax \le b + \left(\frac{n}{k}- 1\right) \left(b + \|a\|_1\right).
				\end{align*}
				
				We claim that $\bar{x}$ violates this inequality. First notice that $a\bar{x} = a\bar{y} + \lambda = b + \lambda > b + 2 \sqrt{n} \left(\frac{n}{k} - 1\right),$ hence it suffices to show $b + \|a\|_1 \le 2 \sqrt{n}$. Our assumption on $\bar{x}$ implies that $\P \neq [-1,1]^n$, and hence $b < \max_{x \in [-1,1]} ax = \|a\|_1$; this gives $b + \|a\|_1 \le 2 \|a\|_1 \le 2 \sqrt{n} \|a\|_2 = 2\sqrt{n}$, thus concluding the proof.
			\end{proof}
			
			To prove the second part of Theorem \ref{thm:upperall} consider a point $\bar{x}$ of distance greater than $2 \sqrt{n} (\frac{n}{k} - 1)$ from $\P$; we show $\bar{x} \notin \P^k$. Let $\bar{y}$ be the closest point to $\bar{x}$ in $\P$. Let $a = \bar{x} - \bar{y}$. From Observation \ref{obs:sepClosesPoint} we have that $ax \le a\bar{y}$ is valid for $\P$. Define $H' = \{x \in \R^n: ax \le a\bar{y}\}$ and $\P' = H' \cap [-1,1]^n$. Notice that $\dist(\bar{x}, H') = \dist(\bar{x}, \bar{y}) > 2 \sqrt{n} (\frac{n}{k} - 1)$. Then Lemma \ref{lemma:sepOneHyp} guarantees that $\bar{x}$ does not belong to ${\P'}^k$. But $\P \subseteq \P'$, so by monotonicity of the $k$-sparse closure we have $\P^k \subseteq \P'^k$; this shows that $\bar{x} \notin \P^k$, thus concluding the proof. 


\section{Lower Bound}\label{sec:lowbnd}

	In this section we prove Theorem \ref{thm:LBRandom01}. The proof is based on the `bad' polytope of Example 2. For a random polytope $\bQ$ in $\R^n$, it is useful to think of each of its (random) faces from the perspective of supporting hyperplanes: for a fixed direction $d \in \R^n$, we have the valid inequality $dx \le \bd_0$, where $\bd_0 = \max_{q \in \bQ} dq$. 
	
	The idea of the proof is then to proceed in two steps. First, for a uniformly random 0/1 polytope $\bP$, we show that with good probability the faces $dx \le \bd_0$ for $\bP^k$ have $\bd_0$ being large, namely $\bd_0 \gtrsim \left(\frac{1}{2} + \frac{\sqrt{\log t}}{\sqrt{k}} \right)\sum_i d_i$, forced by some point $p \in \bP$ with large $dp$; therefore, with good probability the point $\bar{p} \approx (\frac{1}{2} + \frac{\sqrt{\log t}}{\sqrt{k}}) e$ belongs to $\bP^k$. In the second step, we show that with good probability the distance from $\bar{p}$ to $\bP$ is at least $\approx \sqrt{\frac{n}{k}} \sqrt{\log t}$, by showing that the inequality $\sum_i x_i \lesssim \frac{n}{2} + \sqrt{n}$ is valid for $\bP$. 
		
		We now proceed with the proof. {\bred Assume the conditions on $k, n, t$ as stated in Theorem \ref{thm:LBRandom01} hold.} Consider the random set {\bred $\bcX$ defined as} $\{\X^1, \X^2, \ldots, \X^t\}$ where the $\X^i$'s are independent uniform random points in $\{0,1\}^n$, and define the random 0/1 polytope $\bP = \conv(\bcX)$. To formalize the preceding discussion, we need the following definition.
{\bred
\begin{definition}
We say that a (deterministic) 0/1 polytope in $\R^n$ is \emph{$\alpha$-tough} if for every facet $dx \le d_0$ of its $k$-sparse closure we have $d_0 \ge \frac{\sum_i d_i}{2} + \frac{\alpha}{2 \sqrt{k}}(1-\frac{1}{k^2}) \|d\|_1 - \|d\|_\infty/2k^2$, for every $k \in \{2,\dots, n\}$. 
\end{definition} 		}
 		The main element of the lower bound is the following anticoncentration result; in our setting, the idea is that for every ($k$-sparse) direction $d \in \R^n$, with good probability we will have a point $p$ in $\bP^k$ (in fact in $\bP$) with large $dp$.
 		
		\begin{lemma} \label{lemma:anticoncentration}
			Let $\bZ_1, \bZ_2, \ldots, \bZ_n$ be independent random variables with $\bZ_i$ taking value 0 with probability 1/2 and value 1 with probability 1/2 {\bred for every $i \in [n]$.} Then for every $a \in [-1,1]^n$ and $\alpha \in [0, \frac{\sqrt{n}}{8}]$, $$\Pr\left(a\bZ \ge \E[a\bZ] + \frac{\alpha}{2 \sqrt{n}} \left(1 - \frac{1}{n^2}\right) \|a\|_1 - \frac{1}{2 n^2} \right) \ge \left(e^{-50 \alpha^2} - e^{-100 \alpha^2}\right)^{60 \log n}.$$ 
		\end{lemma}	 		
 		
 	The proof of this lemma is reasonably simple and proceeds by grouping the random variables with similar $a_i$'s and then applies known anticoncentration to each of these groups; this proof is presented in Section \ref{app:anticoncentration} of the appendix. 

	In order to effectively apply this anticoncentration to \emph{all} valid inequalities/directions of $\bP^k$, we need some additional control. Define $\mathcal{D}\subseteq \mathbb{Z}^n$ as the set of all integral vectors $\ell \in \R^n$ that are $k$-sparse and satisfy {\bred $\|\ell\|_\infty \le (k)^{k/2}$}. 
 		
		\begin{lemma} \label{lemma:facetsPk}
			Let $\Q \subseteq \R^n$ be a 0/1 polytope. Then for every $k \in [n]$, there is a subset $\mathcal{D}' \subseteq \mathcal{D}$ such that $\Q^k = \{x : dx \le \max_{y \in \Q^k} d y, ~d \in \mathcal{D}'\}$.
		\end{lemma}
		
		This lemma follows directly from applying Corollary 26 in \cite{ziegler2000lectures} to each term $\Q + \R^{\bar{I}}$ in the definition of $\Q^k$ from Section \ref{sec:prelim}.
	
		Employing this lemma to each scenario, we get that all the directions of facets of $\bP^k$ come from the set $\mathcal{D}$. This allows us to analyze the probability that $\bP$ is $\alpha$-tough.
		
		\begin{lemma} \label{label:probTough}
			{\bred Assume the conditions on $k, n, t$ as stated in Theorem \ref{thm:LBRandom01} hold.} If {\bred $1 \leq \alpha^2 \le \min\left\{\frac{\log t}{12000 \log n}, \frac{k}{64} \right\}$}, then $\bP$ is $\alpha$-tough with probability at least $1/2$.
		\end{lemma}
		
		\begin{proof}
 Let $\cE$ be the event that for all $d \in \mathcal{D}$ we have $\max_{i \in [t]} d \X^i \ge \frac{1}{2} \sum_j d_j + \frac{\alpha}{2\sqrt{k}} (1 - \frac{1}{k^2}) \|d\|_1 - \|d\|_\infty/2k^2$. Because of Lemma \ref{lemma:facetsPk}, whenever $\cE$ holds we have that $\bP$ is $\alpha$-tough and thus it suffices to show $\Pr(\cE) \ge 1/2$.
			
			Fix $d \in \mathcal{D}$. Since $d$ is $k$-sparse, and $\alpha \le \frac{\sqrt{k}}{8}$, we can apply Lemma \ref{lemma:anticoncentration} to $d/\|d\|_\infty$ restricted to the coordinates in its support to obtain that
			\begin{align*}
				\Pr\left(d\X^i \ge \frac{\sum_i d_i}{2} + \frac{\alpha}{2 \sqrt{k}}\left(1 - \frac{1}{k^2}\right) \|d\|_1 - \frac{\|d\|_\infty}{2k^2}\right) &\ge \left(e^{-50 \alpha^2} - e^{-100 \alpha^2}\right)^{60 \log n} \\
				&\ge e^{-100 \alpha^2 \cdot 60 \log n} \ge \frac{1}{t^{1/2}},
			\end{align*}
			where the second inequality follows from {\bred the lower bound on $\alpha^2$ (in fact $\alpha^2\geq \frac{\log 2}{50}$ is sufficient)} and the last inequality follows from our upper bound on {\bred $\alpha^2$}. By independence of the $\X^i$'s,
			\begin{align*}
				\Pr\left(\max_{i \in [t]} d\X^i < \frac{\sum_i d_i}{2} + \frac{\alpha}{2 \sqrt{k}}\left(1 - \frac{1}{k^2}\right) \|d\|_1 - \frac{\|d\|_\infty}{2 k^2}\right) \le \left(1 - \frac{1}{t^{1/2}}\right)^t \le e^{-t^{1/2}},
			\end{align*}
			where the second inequality follows from the fact that $(1 - x) \le e^{-x}$ for all $x$.
			
			Finally notice that {\bred $|\mathcal{D}| = {n \choose k} \left(2k^{k/2} + 1\right)^k \leq \left(\frac{ne}{k}\right)^k \left(ek^{k/2}\right)^k \leq \left( \frac{e^2nk^{k/2}}{k}\right)^k \leq e^{0.5k^2\log n + 2k}$, where the inequalities are based on the fact that $k \geq 2$.} By our assumption on the size of $t$ and $k$, we therefore have {\bred $e^{-t^{1/2}}|\mathcal{D}| \le (1/2) $}. Therefore, taking a union bound over all $d \in \mathcal{D}$ of the previous displayed inequality gives $\Pr(\cE) \ge 1/2$, concluding the proof of the lemma. 
		\end{proof}
		
		The next lemma takes care of the second step of the argument.
		
		\begin{lemma} \label{lemma:PSmall}
			With probability at least $3/4$, the inequality $\sum_j x_j \le \frac{n}{2} + 3 \sqrt{n \log t}$ is valid for $\bP$.
		\end{lemma}		
		
		\begin{proof}
			Fix an $i \in [t]$. Since $\Var(\X^i) = n/4$, we have from Bernstein's inequality
			{\bred
			\begin{eqnarray*}
				\Pr\left(\sum_j \X^i_j > \frac{n}{2} + 3 \sqrt{n \log t}\right) &\le& \exp\left(- \min\left\{9 \log t, \frac{9 \sqrt{n \log t}}{4}\right\}  \right)\\
														& \le & e^{- \frac{9 \log t}{4}} \le \frac{1}{4t},
			\end{eqnarray*}}
			where the second inequality follows from the fact that $\log t \le n$, and the last inequality uses the fact that $t \ge 4$. Taking a union bound over all $i \in [t]$ gives 
			{\bred
			\begin{align*}
				\Pr\left(\bigvee_{i \in [t]} \left(\sum_j \X^i_j > \frac{n}{2} + 3 \sqrt{n \log t}\right)\right) \le \frac{1}{4},
			\end{align*}			}
		Finally, notice that an inequality $dx \le d_0$ is valid for $\bP$ iff it is valid for all $\X^i$. This concludes the proof. 
		\end{proof}
		
		\begin{lemma} \label{lemma:toughImpliesDist}
			Suppose that the polytope $\Q$ is $\alpha$-tough for $\alpha \ge 1$ and that the inequality $\sum_i x_i \le \frac{n}{2} + 3 \sqrt{n \log t}$ is valid for $\Q$. Then we have $\dist(\Q, \Q^k) \ge \sqrt{n} \left( \frac{\alpha}{2 \sqrt{k}} - \frac{\alpha}{k^2} - \frac{3 \sqrt{\log t}}{\sqrt{n}}\right)$.
		\end{lemma}
	 	
	 	\begin{proof}
	 		We first show that the point $\bar{q} =  (\frac{1}{2} + \frac{\alpha}{2 \sqrt{k}} - \frac{\alpha}{k^2}) e$ belongs to {\bred $\Q^k$.} Let $dx \le d_0$ be a facet for {\bred $\Q^k$.} Then we have 
			{\bred
	 		\begin{align*}
	 			d \bar{q} &= \frac{\sum_i d_i}{2} + \alpha \left( \frac{1}{2 \sqrt{k}} - \frac{1}{k^2}\right) \sum_i d_i \le \frac{\sum_i d_i}{2} + \alpha \left( \frac{1}{2 \sqrt{k}} - \frac{1}{k^2}\right) \|d\|_1 \\
	 			&\le \frac{\sum_i d_i}{2} + \alpha \left( \frac{1}{2 \sqrt{k}} - \frac{1}{2 k^2}\right) \|d\|_1 - \frac{\|d\|_\infty}{2 k^2}\\
				&\le \frac{\sum_i d_i}{2} + \frac{\alpha}{2 \sqrt{k}}\left(1 - \frac{1}{k^2}\right)\|d\|_1 - \frac{\|d\|_\infty}{2 k^2},
	 		\end{align*}}
	 		where the first inequality uses the fact that $\frac{1}{2 \sqrt{k}} - \frac{1}{k^2} \ge 0$ {\bred for $k \geq 2$} and the second inequality uses $\alpha \ge 1$ and $\|d\|_1 \ge \|d\|_\infty$. Since $\Q$ is $\alpha$-tough it follows that $\bar{q}$ satisfies $dx \le d_0$; since this holds for all facets of {\bred $\Q^k$}, we have {\bred $\bar{q} \in \Q^k$.}
	 		
	 		Now define the halfspace $H = \{ x : \sum_i x_i \le \frac{n}{2} + 3 \sqrt{n \log t}\}$. By assumption $\Q \subseteq H$, and hence $\dist(\Q, \Q^k) \ge \dist(H, \Q^k$). But it is easy to see that the point in $H$ closest to $\bar{q}$ is the point $\tilde{q} = (\frac{1}{2} + \frac{3 \sqrt{\log t}}{\sqrt{n}}) e$. This gives that $\dist(\Q, \Q^k) \ge \dist(H, \Q^k) \ge \dist(\bar{q}, \tilde{q}) \ge \sqrt{n} \left( \frac{\alpha}{2 \sqrt{k}} - \frac{\alpha}{k^2} - \frac{3 \sqrt{\log t}}{\sqrt{n}}\right)$. This concludes the proof. 
	 	\end{proof}

		We now conclude the proof of Theorem \ref{thm:LBRandom01}. 

		\begin{proof}\emph{of Theorem \ref{thm:LBRandom01}} Set $\bar{\alpha}^2 = \min\left\{\frac{\log t}{12000 \log n}, \frac{k}{64} \right\}$. Taking union bound over Lemmas \ref{label:probTough} and \ref{lemma:PSmall}, with probability at least $1/4$, $\bP$ is $\bar{\alpha}$-tough and the inequality inequality $\sum_i x_i \le \frac{n}{2} + 3 \sqrt{n \log t}$ is valid for it. Then from Lemma \ref{lemma:toughImpliesDist} we get that with probability at least $1/4$, $\dist(\bP, \bP^k) \ge \sqrt{n} \left( \frac{\bar{\alpha}}{2 \sqrt{k}} - \frac{\bar{\alpha}}{k^2} - \frac{3 \sqrt{\log t}}{\sqrt{n}}\right)$, and the result follows by plugging in the value of $\bar{\alpha}$. 
\end{proof}



	\section{Hard Packing Integer Programs}
	
	In this section we prove Theorem \ref{thm:PIPLB}. With overload in notation, we use ${[n] \choose k}$ to denote the set of vectors in $\{0,1\}^n$ with exactly $k$ 1's. 
	
	Let $\bP$ be a random polytope sampled from the distribution $(n,m,M)$-PIP and consider the corresponding random vectors $\A^j$'s. The idea of the proof is to show that with constant probability $\bP$ behaves like Example 2, by showing that the cut $\sum_i x_i \lesssim \frac{n}{2}$ is valid for it and that $\bP$ approximately contains 0/1 points with many 1's. Then we show that this `approximate containment' implies that a point with a lot of mass (say, $\approx (1, 1, \ldots, 1)$ for $k \le n/2$) belongs to the $k$-sparse closure $\bP^k$; since such point is far from hyperplane $\sum_i x_i \lesssim \frac{n}{2}$, it is also far from $\bP$ and hence we get a lower bound on $\dist(\bP, \bP^k)$.

	The first part of the argument is a straightforward application of Bernstein's inequality and union bound; its proof is presented in Section \ref{app:hardPIP} of the appendix. 
	
	\begin{lemma} \label{lemma:PIPvalidCut}
		With probability at least $1 - \frac{1}{4}$ the cut $(1 - \frac{2 \sqrt{\log 8n}}{\sqrt{m}}) \sum_i x_i \le \frac{n}{2} + \frac{\sqrt{n \log 8}}{\sqrt{m}}$ is valid for $\bP$. 
	\end{lemma}	

	The other steps in the argument are more involved.
	

	\subsection{Approximate Containment of Points with Many 1's}
	
	First we control the right-hand side of the constraints $\A^j x \le \frac{\sum_i \A^j_i}{2}$ that define $\bP$, by showing that they are roughly $\frac{nM}{2}$; this is again a straightforward application of Bernstein's inequality and is also deferred to Section \ref{app:hardPIP} of the appendix.
	
	\begin{lemma} \label{lemma:PIPRHS}
		With probability at least $1- \frac{1}{8}$ we have $|\sum_{i = 1}^n \A^j_i - \frac{n M}{2}| \le M \sqrt{n \log 8 m}$ for all $j \in [m]$. 
	\end{lemma}	

	Recall that we defined $c = \frac{k}{n}$. Now we show that with constant probability, \emph{all} points $\bar{x} \in \{0,1\}^n$ with $cn$ 1's satisfy {\bred $\A^j\bar{x} \lesssim \frac{nM}{2}$} for all $j \in [m]$, and hence they approximately belong to $\bP$. The argument is cleaner is the random variables $\A^j_i$ were uniformly distributed in the \emph{continuous} interval $[0,M]$, instead of on the discrete set $\{0, \ldots, M\}$;  this is because in the former we can leverage the knowledge of the \emph{order statistics} of continuous uniform variables. Our next lemma then essentially handles this continuous case.
	
	\begin{lemma} \label{lemma:orderStat}
		Let $\U \in \R^n$ be a random variable where each coordinate $\U_i$ is independently drawn uniformly from $[0,1]$. Then with probability at least $1 - 1/8m$ we have $\U\bar{x} \le \frac{c (2n - cn + 1)}{2} + \sqrt{10cnm}$ for all vectors $\bar{x} \in {[n] \choose cn}$.  
	\end{lemma}
	
	\begin{proof}
		Let $\U_{(i)}$ be the $i$th order statistics of $\U_1, \U_2, \ldots, \U_n$ (i.e. in each scenario $U_{(i)}$ equals the $i$th smallest value among $\U_1, \U_2, \ldots, \U_n$ in that scenario). Notice that $\max_{\bar{x} \in {n \choose cn}} \U\bar{x} = \U_{(n)} + \ldots + \U_{(n - cn + 1)}$, and hence is it equivalent to show that $$\Pr\left(\U_{(n)} + \ldots + \U_{(n - cn + 1)} > \frac{c (2n - cn + 1)}{2} + \sqrt{10cnm}\right) \le \frac{1}{8m}.$$ We use $\bZ \triangleq \U_{(n)} + \ldots + \U_{(n - cn + 1)}$ to simplify the notation. 
		
		It is known that $\E[\U_{(i)}] = \frac{i}{n+1}$ and $\Cov(\U_{(i)}, \U_{(j)}) = \frac{i (n + 1 - j)}{(n + 1)^2 (n+2)} \le \frac{1}{n}$ \cite{orderStatistics}. Also, since $\U_{(i)}$ lies in $[0,1]$, we have $\Var(\U_{(i)}) \le 1/4$. Using this information, we get $\E[\bZ] = \frac{(2n - cn + 1) cn}{2 (n+1)} \le \frac{c (2n - cn + 1)}{2}$ and $$\Var(\bZ) \le \frac{cn}{4} + \frac{(cn)^2}{n} \le \frac{5cn}{4},$$ where the last inequality follows from the fact $c \le 1$. Then applying Chebychev's inequality \cite{kol}, we get $$\Pr\left(\bZ \ge \frac{c (2n - cn + 1)}{2} + \sqrt{10 c n m }\right) \le \frac{\Var(\bZ)}{10 cnm} \le \frac{1}{8m}.$$ This concludes the proof. 
	\end{proof}

	Now we translate this proof from the continuous to the discrete setting. 	
	
	\begin{lemma} \label{lemma:PIPAlmostFeasible}
		 With probability at least $1-\frac{1}{8}$ we have 
		 \begin{align*}
		 	\A^j \bar{x} \le \frac{(M+1) c(2n-cn+1)}{2} + (M+1) \sqrt{10 cnm}, ~~~~~~\forall j \in [m], \forall \bar{x} \in {[n] \choose cn}.
		 \end{align*}
	\end{lemma}	

	\begin{proof}
		For each $j \in [m]$, let $\U_1^j, \U_2^j, \ldots, \U^j_n$ be independent and uniformly distributed in $[0,1]$. Define $\Y_i^j \triangleq \lfloor (M+1) \U^j_i \rfloor$. Notice that the random variables $(\Y^j_i)_{i,j}$ have the same distribution as $(\A^j_i)_{i,j}$. So it suffices to prove the lemma for the variables $\Y^j_i$'s. 
		
		Fix $j \in [m]$. For any $\bar{x} \in \{0,1\}^n$ we have $\Y^j\bar{x} \le (M + 1) \U \bar{x}$. Therefore, from Lemma \ref{lemma:orderStat} we get 
		\begin{align*}
			\Pr\left(\bigvee_{\bar{x} \in {n \choose cn}} \left(\Y^j \bar{x} > \frac{(M+1) c(2n - cn + 1)}{2} + (M+1) \sqrt{10 cnm}\right)   \right) \le \frac{1}{8m}.
		\end{align*}
		
		Taking a union bound of this last expression over all $j \in [m]$ concludes the proof of the lemma. 
	\end{proof}

	
	\subsection{From Approximate to Actual Containment}

	From the previous section we get with constant probability, points $\bar{x} \in \{0,1\}^n$ with $cn$ 1's approximately belong to $\bP$; thus, scaling them by a small factor, shows that these points belong to the \emph{LP relaxation} of $\bP$. {\bred Our goal is to strengthen this result by showing that a small (although slightly larger) scaling of these point actually brings them to the integer hull $\bP$ \emph{itself}}.
	
	The next lemma shows that this is in fact possible. 
	
	\begin{lemma} \label{lemma:PIPFracFeasible}
		Consider a 0/1 polytope $\Q = \conv(\{ x \in \{0,1\}^n : a^j x \le b_j, ~j=1, 2, \ldots, m\})$ where {\bred $n \ge 50$,} $m \le n$, $a^j_i \in [0, M]$ for all $i, j$, and $b_j \ge \frac{nM}{12}$ for all $i$. Consider $1 < \alpha \le 2 \sqrt{n}$ and let $\bar{x} \in \{0,1\}^n$ be such that for all $j$, $a^j \bar{x} \le \alpha b_j$. Then the point $\frac{1}{\alpha}(1-\epsilon)^2 \bar{x}$ belongs to $\Q$ as long as $\frac{12 \sqrt{\log 4 n^2 m}}{\sqrt{n}} \le \epsilon \le \frac{1}{2}$.
	\end{lemma}
	
	For the remainder of the section we prove this lemma. The idea is that we can select a subset of $\approx 1-1/\alpha$ coordinates and change $\bar{x}$ to 0 in these coordinates to obtain a feasible solution in $\Q$; repeating this for many sets of coordinates and taking an average of the feasible points obtained will give the result. 
		
		To make this precise, let $p = \frac{1}{\alpha}(1-\epsilon)$. For $w \in [n^2]$ define the independent random variables $\X^w_1, \X^w_2, \ldots, \X^w_n$ taking values in $\{0,1\}$ such that $\E[\X^w_i] = p \bar{x}_i$ (i.e. if $\bar{x}_i = 1$, then keep it at 1 with probability $p$, otherwise flip it to 0; if $\bar{x}_i = 0$, then keep it at 0).
		
		\begin{claim}
			With probability at least $3/4$ all points $\X^w$ belong to $\Q$.
		\end{claim} 
		
		\begin{proof}
			Notice {\bred $\E[a^j \X^w] \leq (1 - \epsilon) b_j$.} Also, from our upper bound on $a^j$, we have $\Var(a^j \X^w) \le \frac{M^2 n}{4}$. Employing Bernstein's inequality, 
			\begin{align*}
				\Pr(a^j \X^w > b_j) \le \exp\left(-\min\left\{\frac{\epsilon^2 b_j^2}{M^2 n}, \frac{3 \epsilon b_j}{4 M} \right\} \right) \le \frac{1}{4 n^2 m},
 			\end{align*}	
 			where the second inequality uses the assumed lower bounds on $b_j$ and $\epsilon$, and the fact that $\frac{4\cdot 12 \log 4n^2m}{3n} \le \frac{12 \sqrt{\log 4n^2 m}}{\sqrt{n}}$ due to our bounds on $n$ and $m$. The claim follows by taking a union bound over all $j$ and $w$. 
		\end{proof}	
		
		Let $\bZ = \frac{1}{n^2} \sum_w \X^w$ be the random point that is the average of the $\X^w$'s. 
		
		\begin{claim}
			With probability at least $3/4$, $\bZ_i \ge \frac{1}{\alpha}(1-\epsilon)^2 \bar{x}_i$ for all $i$.  
		\end{claim}
		
		\begin{proof}
			Since $\bar{x} \in \{0,1\}^n$, it suffices to consider indices $i$ such that $\bar{x}_i = 1$. Fix such an $i$. We have $\E[n^2 \bZ_i] = p n^2$ and $\Var(n^2 \bZ_i) \le \frac{n^2}{4}$. Then from Bernstein's inequality 
			\begin{align*}
				\Pr\left(\bZ_i < \frac{1}{\alpha}(1-\epsilon)^2 \bar{x}_i\right) &= \Pr(n^2 \bZ_i < \E[n^2 \bZ_i] (1 - \epsilon)) \\
				&\le \exp\left(- \min\left\{n^2 (\epsilon p)^2, \frac{3 n^2 \epsilon p }{4} \right\}  \right) \le \frac{1}{4n},
			\end{align*}  
			where the last inequality uses the lower bound on $\epsilon$, the fact that $n \ge 50$, and the fact that $p \ge 1/2\alpha \ge 1/4 \sqrt{n}$. The claim follows from taking a union bound over all $i$ such that $\bar{x}_i = 1$. 
		\end{proof}

	Taking a union bound over the above two claims we get that there is a realization $\tilde{x}^1, \tilde{x}^2, \ldots, \tilde{x}^{n^2}$ of the random vectors $\X^1, \X^2, \ldots, \X^{n^2}$ such that (let $\tilde{z} = \frac{1}{n^2} \sum_w \tilde{x}^w$): (i) All $\tilde{x}^w$ belong to $\Q$, and hence so does their convex combination $\tilde{z}$; (ii) $\tilde{z} \ge \frac{1}{\alpha} (1-\epsilon)^2 \bar{x}$. Since $\Q$ is of packing-type, it follows that the point $\frac{1}{\alpha} (1-\epsilon)^2 \bar{x}$ belongs to $\Q$, concluding the proof of Lemma \ref{lemma:PIPFracFeasible}.


	\subsection{Proof of Theorem \ref{thm:PIPLB}}

		Now we put together the {\bred results from the} previous sections to conclude the proof of Theorem \ref{thm:PIPLB}.	
Let $\cE$ be the event that Lemmas \ref{lemma:PIPRHS}, \ref{lemma:PIPvalidCut} and  \ref{lemma:PIPAlmostFeasible} hold; notice that $\Pr(\cE) \ge 1/2$. For the rest of the proof we fix a $\bP$ (and the associated $\A^j$'s) where $\cE$ holds and prove a lower bound on $\dist(\bP, \bP^k)$.
			
		Consider a set $I \in {[n] \choose cn}$ and let $\bar{x}$ be the incidence vector of $I$ (i.e. $\bar{x}_i = 1$ if $i \in I$ and $\bar{x}_i = 0$ if $i \notin I$). Since the bounds from Lemmas \ref{lemma:PIPRHS} and \ref{lemma:PIPAlmostFeasible} hold for our $\bP$, straightforward calculations show that $\A^j \bar{x} \le \alpha \frac{1}{2} \sum_i \A^j_i$ for all $j \in [m]$. Therefore, from Lemma \ref{lemma:PIPFracFeasible} we have that the point $\frac{1}{\max\{\alpha, 1\}} (1- \epsilon)^2 \bar{x}$ belongs to $\bP$. This means that the point $\tilde{x} = \frac{1}{\max\{\alpha, 1\}} (1- \epsilon)^2 e$ belongs to $\bP + \R^{\bar{I}}$ (see Section \ref{sec:prelim}). Since this holds for every $I \in {[n] \choose cn}$, we have $\tilde{x} \in \bP^k$.
		
		Let $\tilde{\y}$ be the point in $\bP$ closest to $\tilde{x}$. Let $a = (1 - \frac{2 \sqrt{\log 8n}}{\sqrt{m}})$ and $b = \frac{n}{2} + \sqrt{n \log 8 m}$, so that the cut in Lemma \ref{lemma:PIPvalidCut} is given by $aex \le b$. From Cauchy-Schwarz we have that $\dist(\tilde{x}, \tilde{\y}) \ge \frac{ae \tilde{x} - ae \tilde{\y}}{\|ae\|} = \frac{e \tilde{x}}{\sqrt{n}} - \frac{ae \tilde{\y}}{a \sqrt{n}}$. 
		
		By definition of $\tilde{x}$ we have $e \tilde{x} = \frac{1}{\max\{\alpha, 1\}} (1- \epsilon)^2 n$. From the fact the cut $aex \le b$ is valid for $\bP$ and $\tilde{\y} \in \bP$, we have $ae \tilde{\y} \le b$. Simple calculations show that $\frac{b}{a \sqrt{n}} \le \frac{n}{2} (1 + \epsilon')$. Plugging these values in we get that $\dist(\bP, \bP^k) = \dist(\tilde{x}, \tilde{\y}) \ge \frac{\sqrt{n}}{2} \left(\frac{2 (1-\epsilon)^2}{\max\{\alpha, 1\}} - (1 + \epsilon')\right)$. Theorem \ref{thm:PIPLB} follows from the definition of $\alpha, \epsilon$ and $\epsilon'$.


	\section{Sparse Cutting-Planes and Extended Formulations} \label{sec:exended}

	In this section we analyze the relationship between sparse cuts and extended formulations, proving Proposition \ref{lemma:extFormContainment} and Proposition \ref{lemma:goodExtFormulation}.
	
	\subsection{Proof of Proposition \ref{lemma:extFormContainment}}
	
		For any set $S \subseteq \R^{n'}$ and $I \subseteq [n']$, define $\tau_I(S) = S + \R^{\bar{I}}$ (recall that $\R^{\bar{I}} = \{x \in \R^{n'} : x_i = 0 \textrm{ for } i \in I\}$. 
	
		Consider $\P \subseteq \R^n$ and $\Q \subseteq \R^n \times \R^m$ such that $\P = \proj_x(\Q)$. Given a subset $I \subseteq [n + m]$ we use $I_x$ to denote the indices of $I$ in $[n]$ (i.e. $I_x = I \cap [n])$. We start with the following technical lemma.	
	
	\begin{lemma}\label{le:commute}
		For every $I \subseteq [n + m]$ we have $\tau_{I_x} (\proj_x(\Q)) = \proj_x(\tau_{I}(\Q))$.
	\end{lemma}
	
	\begin{proof}
		$(\subseteq)$ Take $u_x \in \tau_{I_x}(\proj_x(\Q)$; this means that there is $v \in \Q$ such that $u_x = \proj_x(v) + d_x$ for some vector $d_x \in \R^n$ with support in $I_x$. Define $d = (d_x, 0) \in \R^n \times \R^m$, with support in $I_x \subseteq I$. Then $v + d$ belongs to $\tau_I(\Q)$ and $$u_x = \proj_x(v) + d_x = \proj_x(v + d) \in \proj_x(\tau_I(\Q)),$$ concluding this part of the proof.
		
		$(\supseteq)$ Take $u_x \in \proj_x(\tau_I(\Q))$. Let $u \in \tau_I(\Q)$ be such that $\proj_x(u) = u_x$. By definition, there is $d \in \R^n \times \R^m$ with support in $I$ such that $u + d$ belongs to $\Q$. Then $\proj_x(u + d) = u_x + \proj_x(d)$ belongs to $\proj_x(\Q)$; since $\proj_x(d)$ is supported in $I_x$, we have that $u_x$ belongs to $\tau_{I_x}(\proj_x(\Q))$, thus concluding the proof of the lemma. 
	\end{proof}

	The proof of Proposition \ref{lemma:extFormContainment} then follows directly from the above lemma:
	\begin{align*}
		(\proj_x(\Q))^k &= \bigcap_{J \subseteq [n]} \tau_{J} (\proj_x(\Q)) = \bigcap_{I \subseteq [n+m]} \tau_{I_x} (\proj_x(\Q)) \\
		&\stackrel{\textrm{Lemma \ref{le:commute}}}{=}  \bigcap_{I \subseteq [n + m]} \proj_x(\tau_{I}(\Q)) \supseteq \proj_x\left(\bigcap_{I \subseteq [n+m]} \tau_{I}(\Q) \right) = \proj_x(\Q^k).
\end{align*}


	\subsection{Proof of Proposition \ref{lemma:goodExtFormulation}}
	
		We construct the polytope $\Q \subseteq \R^{n} \times \R^{2n - 1}$ as follows. Let $T$ be the complete ordered binary tree of height $\ell + 1$. We let $\roott$ denote the root node of $T$. We use $\intt(T)$ to denote the set of internal nodes of $T$, and for an internal node $v \in \intt(T)$ we use $\leftc(v)$ to denote its left child and $\rightc(v)$ to denote its right child. Let $i(.)$ be a bijection between the leaf nodes of $T$ and the elements of $[n]$. We then define the set $\Q$ as the solutions $(x, y)$ to the following:
	\begin{align}
		y_{\roott} &\le 1 \notag\\
		y_v &= y_{\leftc(v)} + y_{\rightc(v)}, ~\forall v \in \intt(T) \notag\\
		y_v &= \frac{2}{n} x_{i(v)}, ~\forall v \in T \setminus \intt(T) \label{eq:goodExt}\\
		y_v &\ge 0, ~ \forall v \in T \notag\\
		x_i &\in [0,1], ~\forall i \in [n]. \notag
	\end{align}	
	
	Define $\P = \{x \in [0,1]^{n} : \sum_{i \in [n]} x_i \le n/2\}$.
	
	\begin{claim}
		$\Q$ is an extended formulation of $\P$, namely $\proj_x(\Q) = \P$.
	\end{claim}
	
	\begin{proof}
		$(\subseteq)$ Take $(\bar{x}, \bar{y}) \in \Q$. Let $T_j$ denote the set of nodes of $T$ at level $j$. It is easy to see (for instance, by reverse induction on $j$) that $\sum_{v \in T_j} \bar{y}_v = \frac{2}{n} \sum_{i \in [n]} \bar{x}_i$ for all $j$. In particular, $\bar{y}_{\roott} = \frac{2}{n} \sum_{i \in [n]} \bar{x}_i$. Since $\bar{y}_{\roott} \le 1$, we have that $\bar{x} \in \P$.
		
		$(\supseteq)$ Take $\bar{x} \in \P$. Define $\bar{y}$ inductively by setting $\bar{y}_{i(v)} = \frac{2}{n} \bar{x}_{i(v)}$ for all leaves $v$ and $\bar{y}_v = \bar{y}_{\leftc(v)} + \bar{y}_{\rightc(v)}$ for all internal nodes $v$. As in the previous paragraph, it is easy to see that $\bar{y}_{\roott} = \sum_{i \in [n]} \bar{x}_i \le n/2$. Therefore, $(\bar{x}, \bar{y})$ belongs to $\Q$. 
	\end{proof}
	
	\begin{claim}
		$\dist(\P, \P^k) = \sqrt{n/2}$ for all $k \le n/2$.
	\end{claim}
	
	\begin{proof}
		For every subset $I \subseteq [n]$ of size $n/2$, the incidence vector of $I$ belongs $\P$ this implies that, when $k \le n/2$, the all ones vector $e$ belongs to $\P^k$. It is easy to see that the closest vector in $\P$ to $e$ is the vector $\frac{1}{2} e$; since the distance between $e$ and $\frac{1}{2} e$ is $\sqrt{n/2}$, the claim follows.
	\end{proof}
	
	\begin{claim}
		$\Q^3 = \Q$.
	\end{claim}
	
	\begin{proof}
		Follows directly from the fact that all the equations and inequalities defining $\Q$ in \eqref{eq:goodExt} have support of size at most 3. 
	\end{proof}
	
	The proof of Proposition \ref{lemma:goodExtFormulation} follows directly from the three claims above.



\bibliographystyle{abbrv}      
\bibliography{sparse_cuts_Arxiv}   

\appendix
\normalsize

		\section{Concentration Inequalities} \label{app:bernstein}
		
		We state Bernstein's inequality in a slightly weaker but more convenient form. 
		
		\begin{theorem}[Bernstein's Inequality {[\cite{kol}, Appendix A.2]}] \label{thm:bernsteins}
			Let $\X_1, \X_2, \ldots, \X_n$ be independent random variables such that $|\X_i - \E[\X_i]| \le M$ for all $i \in [n]$. Let $\X = \sum_{i = 1}^n \X_i$ and define $\sigma^2 = \Var(\X)$. Then for all $t > 0$ we have
			\begin{align*}
				\Pr(|\X - \E[\X]| > t) \le \exp\left(-\min\left\{\frac{t^2}{4 \sigma^2}, \frac{3t}{4M} \right\} \right).
			\end{align*}  
		\end{theorem}

\section{Empirically Generating Lower Bound on $d(\P,\P^k)$} \label{app:empiricalLB}

We estimate a lower bound on $d(\P,\P^k)$ using the following procedure. The input to the procedure is the set of points $\{p^1, \dots, p^t \} \in [0,1 ]^n$ which are vertices of $\P$. For every $I \in {[n] \choose k}$, we use PORTA to obtain an inequality description of $\P + \mathbb{R}^{\bar{I}}$. Putting all these inequalities together we obtain an inequality description of $\P^k$. Unfortunately due to the large number of inequalities, we are unable to find the vertices of $\P^k$ using PORTA. Therefore, we obtain a lower bound on $d(\P,\P^k)$ via a \emph{shooting experiment}.

First observe that given $u \in \mathbb{R}^n\setminus \{0\}$ we obtain a lower bound on $d(\P,\P^k)$ as: $$\frac{1}{\|u\|}\left(\max\{ u^Tx : x \in \P^k\} - \max\{ u^Tx : x \in \P\} \right).$$
Moreover it can be verified that there exists a direction which achieves the correct value of $d(\P, \P^k)$. We generated 20,000 random directions $u$ by picking them uniformly in the set $[-1,1]^n$. 
Also we found that for instances where $p^j \in \{ x \in \{0,1 \}^n \,:\, \sum_{i = 1}^n x_i = \frac{n}{2}\}$, the directions $(\frac{1}{\sqrt{n}}, \dots, \frac{1}{\sqrt{n}})$ and $-(\frac{1}{\sqrt{n}}, \dots, \frac{1}{\sqrt{n}})$ yield good lower bounds. The Figure in Section 1.3(c) plots the best lower bound among the 20,002 lower bounds found as above.


		\section{Anticoncentration of Linear Combination of Bernoulli's} \label{app:anticoncentration}
		
		It is convenient to restate Lemma \ref{lemma:anticoncentration} in terms of Rademacher random variables (i.e. that takes values -1/1 with equal probability).
		
		\begin{lemma}[Lemma \ref{lemma:anticoncentration}, restated] \label{lemma:anticoncentrationRad}
			Let $\X_1, \X_2, \ldots, \X_n$ be independent Rademacher random variables. Then for every $a \in [-1,1]^n$, $$\Pr\left(a\X \ge \frac{\alpha}{\sqrt{n}} \left(1 - \frac{1}{n^2}\right) \|a\|_1 - \frac{1}{n^2} \right) \ge \left(e^{-50 \alpha^2} - e^{-100 \alpha^2}\right)^{60 \log n}, ~~~~~~~~ \alpha \in \left[0, \frac{\sqrt{n}}{8} \right].$$ 
		\end{lemma}		

		We start with the case where the vector $a$ has all of its coordinates being similar.
	
		\begin{lemma} \label{lemma:anticoncentrationRadSimple}
			Let $\X_1, \X_2, \ldots, \X_n$ be independent Rademacher random variables. For every $\epsilon \ge 1/20$ and $a \in [1 - \epsilon, 1]^n$, $$\Pr\left(a\X \ge \frac{\alpha}{\sqrt{n}} \|a\|_1\right) \ge e^{-50 \alpha^2} - e^{-\frac{\alpha^2}{4 \epsilon^2}}, ~~~~~~~~ \alpha \in \left[0, \frac{\sqrt{n}}{8} \right].$$			
		\end{lemma}
		
		\begin{proof}
			Since $a\X = \sum_i \X_i - \sum_i (1-a_i) \X_i$, having $\sum_i \X_i \ge 2t$ and $\sum_i (1-a_i) \X_i \le t$ implies that $a\X \ge t$. Therefore, 
			\begin{align}
				\Pr(a\X \ge t) &\ge \Pr\left(\left(\sum_i \X_i \ge 2t\right) \vee \left(\sum_i (1 -a_i) \X_i \le t\right) \right) \notag\\
				&\ge \Pr\left(\sum_i \X_i \ge 2t\right) - \Pr\left(\sum_i (1 -a_i) \X_i \le t\right), \label{eq:anticoncentrationRad1}
			\end{align}	
			where the second inequality comes from union bound. For $t \in [0, n/8]$, the first term in the right-hand side can be lower bounded by $e^{-\frac{50 t^2}{n}}$ (see for instance Section 7.3 of \cite{matousekVondrak}). The second term in the right-hand side can be bounded using Bernstein's inequality: given that $\Var(\sum_i (1 -a_i) \X_i) = \sum_i (1 - a_i)^2 \le n \epsilon^2$, we get that for all $t \in [0,n/8]$
			\begin{align*}
				\Pr\left(\sum_i (1 -a_i) \X_i \le t\right) \le \exp\left(- \min\left\{\frac{t^2}{4 n \epsilon^2}, \frac{3t}{4 \epsilon} \right\}  \right) = e^{-\frac{t^2}{4n\epsilon^2}}.
			\end{align*}
			The lemma then follows by plugging these bounds on \eqref{eq:anticoncentrationRad1} and using $t = \alpha \sqrt{n} \ge \frac{\alpha}{\sqrt{n}}\|a\|_1$. 
		\end{proof}
		
		\begin{proof}[Proof of Lemma \ref{lemma:anticoncentrationRad}]
			Without loss of generality assume $a > 0$, since flipping the sign of negative coordinates of $a$ does not change the distribution of $a\bZ$ neither the term $\frac{\alpha}{\sqrt{n}} \left(1 - \frac{2}{n^2}\right)\|a\|_1$. Also assume without loss of generality that $\|a\|_\infty = 1$. The idea of the proof is to bucket the coordinates such that in each bucket the values of $a$ is within a factor of $(1 \pm \epsilon)$ of each other, and then apply Lemma \ref{lemma:anticoncentrationRadSimple} in each bucket.
			
			The first step is to trim the coefficients of $a$ that are very small. Define the trimmed version $b$ of $a$ by setting $b_i = a_i$ for all $i$ where $a_i \ge 1/n^3$ and $b_i = 0$ for all other $i$. We first show that 
			\begin{align}
				\Pr\left(b\bZ \ge \frac{\alpha}{\sqrt{n}} \|b\|_1\right) \ge \left(e^{-50 \alpha^2} - e^{-100 \alpha^2}\right)^{60 \log n}, \label{eq:anticoncentrationRad2}
			\end{align}
			and then we argue that the error introduced by considering $b$ instead of $a$ is small. 
			
			For $j \in \{0, 1, \ldots, \frac{3 \log n}{\epsilon}\}$, define the $j$th bucket as $I_j = \{i : b_i \in ((1-\epsilon)^{j+1}, (1-\epsilon)^j]\}$. Since $(1-\epsilon)^{\frac{3 \log n}{\epsilon}} \le e^{-3 \log n} = 1/n^3$, we have that every index $i$ with $b_i > 0$ lies within some bucket. 
			
			Now fix some bucket $j$. Let $\epsilon = 1/20$ and $\gamma = \frac{\alpha}{\sqrt{n}}$. Let $E_j$ be the event that $\sum_{i \in I_j} b_i \bZ_i \ge \gamma \sum_{i \in I_j} b_i$. Employing Lemma \ref{lemma:anticoncentrationRadSimple} over the vector $(1-\epsilon)^j b|_{I_j}$, gives $$\Pr\left(\sum_{i \in I_j} b_i \bZ_i \ge \gamma \sum_{i \in I_j} b_i \right) \ge e^{-50 \gamma^2 |I_j|} - e^{-\frac{\gamma^2 |I_j|}{4 \epsilon^2}} \ge e^{-50 \gamma^2 n} - e^{-\frac{\gamma^2 n}{4 \epsilon^2}} , ~~~~~~~~ \gamma \in \left[0, \frac{1}{8} \right].$$ But now notice that if in a scenario we have $E_j$ holding for all $j$, then in this scenario we have $b\bZ \ge \gamma \|b\|_1$. Using the fact that the $E_j$'s are independent (due to the independence of the coordinates of $\bZ$), we have
			\begin{align*}
				\Pr(b\bZ \ge \gamma \|b\|_1) \ge \Pr\left(\bigvee_j E_j \right) \ge \left(e^{-50 \gamma^2 n} - e^{-\frac{\gamma^2 n}{4 \epsilon^2}}\right)^{\frac{3 \log n}{\epsilon}} , ~~~~~~~~ \gamma \in \left[0, \frac{1}{8} \right].
			\end{align*}
			
			Now we claim that whenever $bX \ge \gamma \|b\|_1$, then we have $a\bZ \ge \frac{\alpha}{\sqrt{n}} \left(1 - \frac{2}{n^2} \right) \|a\|_1$. First notice that $\|b\|_1 \ge \|a\|_1 - 1/n^2 \ge \|a\|_1 (1 - 1/n^2)$, since $\|a\|_1 \ge \|a\|_\infty = 1$. Moreover, with probability 1 we have $a\bZ \ge b\bZ - 1/n^2$. Therefore, whenever $b\bZ \ge \gamma \|b\|_1$: $$a\bZ \ge b\bZ - \frac{1}{n^2} \ge \gamma \|b\|_1 - \frac{1}{n^2} \ge \gamma \left(1 - \frac{1}{n^2}\right) \|a\|_1 - \frac{1}{n^2} =  \frac{\alpha}{\sqrt{n}} \left(1 - \frac{1}{n^2}\right) \|a\|_1  - \frac{1}{n^2}.$$ This concludes the proof of the lemma. 
		\end{proof}


\section{Hard Packing Integer Programs} \label{app:hardPIP}

\subsection{Proof of Lemma \ref{lemma:PIPvalidCut}}
		
		Fix $i \in [n]$. We have $\E[\sum_j \A^j_i] = \frac{mM}{2}$ and $\Var(\sum_j \A^j_i) \le \frac{mM^2}{4}$. Employing Bernstein's inequality we get
		\begin{align*}
			\Pr\left(\sum_j \A^j_i < \frac{mM}{2} - \sqrt{m \log 8n} M\right) \le \exp\left(-\min \left\{\log 8n, \frac{3 \sqrt{m \log 8n}}{4} \right\} \right) \le \frac{1}{8n},
		\end{align*}
		where the last inequality uses the assumption that $m \ge 8 \log 8n$. Similarly, we get that 
		\begin{align*}
			\Pr\left(\sum_{i,j} \A^j_i > \frac{nmM}{2} + \sqrt{n m \log 8n} M\right) \le \exp\left(-\min \left\{\log 8n, \frac{3 \sqrt{n m \log 8n}}{4} \right\} \right) \le \frac{1}{8n}.
		\end{align*}		
		
		Taking a union bound over the first displayed inequality over all $i \in [n]$ and also over the last inequality, with probability at least $1-1/4$ the valid cut $\sum_i (\frac{2}{mM} \sum_j \A^j_i) x_i \le \frac{1}{mM} \sum_{i,j} \A^j_i$ (obtained by aggregating all inequalities in the formulation) has all coefficients on the left-hand side being at least $(1 - \frac{2 \sqrt{\log 8n}}{\sqrt{m}})$ and the right-hand side at most $\frac{n}{2} + \frac{\sqrt{n \log 8}}{\sqrt{m}}$. This concludes the proof. 


	\subsection{Proof of Lemma \ref{lemma:PIPRHS}}

	 	Fix $j \in [m]$. We have $\E[\sum_{i = 1}^n \A^j_i] = \frac{n M}{2}$ and $\Var(	\sum_{i = 1}^n \A^j_i) \le n M^2/4$ and hence by Bernstein's inequality we get 
		\begin{align*}
			\Pr\left(\sum_{i = 1}^n \A^j_i > \frac{n M}{2} + M \sqrt{n \log 8 m}\right) \le \exp\left(- \min \left\{\log 8m, \frac{3 \sqrt{n \log 8 m}}{4} \right\} \right) \le \frac{1}{8m},
		\end{align*}
		 where the last inequality uses the assumption that $m \le n$. The lemma then follows by taking a union bound over all $j \in [m]$.

\end{document}